\documentclass[12pt,a4paper,oneside]{article}
\usepackage[english]{babel}
\usepackage{a4wide,amsmath,amssymb,amsthm}
\usepackage[latin1]{inputenc}
\theoremstyle{definition}
\newtheorem{theorem}{Theorem}
\newtheorem{lemma}[theorem]{Lemma}
\newtheorem{eg}[theorem]{Example}
\newtheorem{example}[theorem]{Example}
\font\tenbfmit=cmmib10
\font\sevenbfmit=cmmib7
\font\fivebfmit=cmmib5
\newcommand{\HH}{\mathcal{H}}
\newcommand{\peq}{\stackrel{\wedge}{=}}
\newcommand{\B}{{\mathcal B}}
\newcommand{\F}{\mathbb{F}}

\author{A. Blokhuis, G. Marino \& F. Mazzocca}
\title{Generalized Hyperfocused Arcs in $PG(2,p)$}
\date{}
\begin{document}
\maketitle
\begin{abstract}
A {\em generalized hyperfocused arc} $\HH $ in $PG(2,q)$
is an arc of size $k$ with the property
that the $k(k-1)/2$ secants can be blocked by a set  of  $k-1$
  points not belonging to the arc. We show that if $q$ is a prime and $\HH$
is a generalized hyperfocused arc of size $k$, then $k=1,2$ or $4$. Interestingly, this problem is also related to the (strong) cylinder conjecture \cite{B1}, as  we point out in the last section.
\end{abstract}
\section{Introduction and preliminaries}
Let $PG(2, q)$ be the projective plane over $F_q,$ the finite field with $q$ elements. A $k-$arc
 in $PG(2, q)$ is a set of $k$ points with no $3$ on a line. A line containing $1$ or $2$
points of a $k-$arc is said to be a tangent or secant to the $k-$arc, respectively.

A {\em blocking set} of a family of lines $ \cal F $ is a point-set
${\cal B} \subset PG(2,q)$ having non-empty intersection with each line in $ \cal F .$ If this is the case, we also say that the lines in $\cal F$ are \em blocked \rm by $\cal B .$

A {\em generalized hyperfocused arc} $\HH $ in $PG(2,q)$
is a  $k$-arc with the property
that the $k(k-1)/2$ secants can be blocked by a set ${\cal B}$ of $k-1$
  points not belonging to the arc.
Points of the arc $\HH$ will be called {\em white points} and points of the blocking
set ${\B}$ {\em black}. In case $k>1,$  since every secant to the arc contains a unique black point,
the $k-1$ black points induce a factorization, i.e. a partition into matchings, of the white $k$-arc and $k$ is forced to be even.
For $k=2,$ we only have a trivial example: $\cal B$ consists of a unique point out of $\HH$ on the line through the two points of $\HH .$

An non trivial example of generalized hyperfocused arc  is any $4$-arc of white points with
its three black diagonal points and our main result is that this is the only non trivial example, provided $q$ is an odd prime.

For $q$
even, there are many examples with all black points on a line; in this case $\cal H$ is simply called
a {\em hyperfocused arc}. As a consequence of the main result of \cite{BK},
hyperfocused arcs only exist if $q$ is even. Their  study  is   motivated by a relevant application to cryptography  in connection with constructions of efficient secret sharing schemes  \cite{H1,S1}.
When $q$ is even, a nice result is that  generalized hyperfocused arcs contained in a conic are hyperfocused  \cite{AKS}; moreover it is known that there exist examples of generalized
hyperfocused arcs which are not hyperfocused \cite{GM}. However, although much more is known about hyperfocused arcs,
there are still many open problems concerning them \cite{AKS,CH,GM}.

\bigskip

Let us conclude this section recalling some results on so-called dual $3-$nets that will play a crucial role in the proof of our main theorem. A {\em{dual $3$-net embedded in}} $PG(2,{\F}),$ the projective plane over a field $\F ,$ is a triple $\{A,B,C\}$ with $A,B,C$ pairwise disjoint point-sets of size $n$, called {\em{components}}, such that every line meeting two distinct components meets each component in precisely one point.

\begin{example} \label{ex1}
Let $\ell$ and $\Gamma $ be a line and a non singular conic in $PG(2,\F)$, respectively. It is well known that an abelian group G  on  $\Gamma\setminus \ell$ can be defined in the following way.
Choose a point $O\in \Gamma\setminus \ell$  as the identity of the group and, for any two points $P, Q\in \Gamma\setminus \ell$,  let $R'$ be the point
 that the line through $P, Q$ has in common with $\ell$. Then the sum of $P$
and $Q$ is defined by $P + Q = R$, where $R$ is the second of the two points (counted with multiplicity) common to the line $OR'$ and $\Gamma\setminus \ell$. Now, given a proper subgroup $A$ of $G$ of finite order $n$  and one of its cosets $B$ ($\not= A$),  the set $C$ of the points of $\ell$ on some line intersecting both $A$ and $B$ has exactly $m$ points. Then the triple $\{A,B,C\}$ is a dual $3-$net of order $n$ embedded in $PG(2,\F)$.
\end{example}

The following theorem follows from the main result of \cite{BKM} (the case that $A\cup B$ is contained
in the union of two lines is also characterized there).

\begin{theorem} \label{BKM}
Let $\{A, B, C\}$ be a dual $3$-net of order $n$ in $PG(2,{\F})$. Then,  if $C$ is contained in a line and ${\F}$ has positive characteristic $p\ge n$, $A\cup B$ is contained in a conic. If this conic is irreducible then
it is of type described in Example \ref{ex1}.
\label{BKM}\end{theorem}
\section{The problem and the main result}
In $PG(2,p)=PG(2,\F),$ $p$ a prime, let us fix a projective frame, so that every point has homogeneous coordinates $(x:y:z)$.
If $A=(a_1,a_2,a_3)$ is a non-zero vector of $\F^3$, we denote by $[A]=\langle(a_1,a_2,a_3)\rangle$
the point of $PG(2,p)$ with homogeneous coordinates $(a_1:a_2:a_3)$. Sometimes, abusing notation,
we will just write $A$ instead of $[A]$ and, in this case, we mean that for the point $[A]$ we are considering the coordinates  $(a_1,a_2,a_3)$ or some other special ones, that should be clear from the context. With this notation we write $A\peq B$ if $[A]=[B]$, i.e. $A=\lambda B$, for some non zero $\lambda$ in $\F$.

Throughout this section $\HH = \{ W_1,W_2,\dots ,W_{2n} \}$ will denote a
  generalized hyperfocused arc of order $2n$ in $PG(2,p),$  with black point-set ${\cal B}.$
$W$ will stand for a projective point, and $E$, with $W=[E]$ a corresponding suitably
chosen representing vector.
If $W_i,W_j$, with $W=[E]$ for some suitable representative $E$,
are two white points, we denote by $B_{ij}$ the unique black point on the line  $<W_i,W_j>$ and we define $b_{ij}$ by $$B_{ij}\peq E_i+b_{ij}E_j.$$ Obviously, since $B_{ji}=B_{ij}$ we have $b_{ji}=1/b_{ij}$.

From now on, \bf we will assume $p\not=2,3$ \rm since our results are trivial in these two cases.

\begin{lemma}
For $i\ne j\ne k\ne i$ we have
\begin{equation}
b_{ij}b_{jk}b_{ki}=1 \label{a1}
\end{equation}
\end{lemma}
\begin{proof}
The proof makes essentially use of the computational technique of Segre's lemma of tangents.
For a (cyclicly ordered) white triple $W_i(=[E_i])$, $W_j$, $W_k$
define the white products: $\omega_i=\prod w_j/w_k$, (and $\omega_j$, $\omega_k$ cyclicly)
with one factor for every white point $W\peq w_iE_i+w_jE_j+w_kE_k$ different from
$W_i$, $W_j$ and $W_k$ (since the white points form an arc $w_i,w_j$ and $w_k\ne0$).
We obviously have $\omega_i\omega_j\omega_k=1$.
Define black products
$\beta_i=\prod b_j/b_k$, (and cyclicly) with one factor for every black point $B\peq b_iE_i+b_jE_j+b_kE_k$,
not on one of the sides of the triangle $W_iW_jW_k$, so this excludes the points
$B_{jk}$, $B_{ki}$ and $B_{ij}$. We again
have $\beta_i\beta_j\beta_k=1$. Now $\omega_i=\beta_i/b_{jk}$, because for every
white point $W\peq w_iE_i+w_jE_j+w_kE_k$ there is a (unique) black point
$B\peq b_iE_i+b_jE_j+b_kE_k$ (on the line $\langle W,W_i\rangle$) with $w_j/w_k=b_j/b_k$, and all these
black points contribute to the product defining $\beta_i$, except $B_{jk}\peq E_j+b_{jk}E_k$ because that
is on a side of the triangle. It follows that $\omega_i=\beta_i/b_{jk}$, and with it $b_{ij}b_{jk}b_{ki}=1$.
\end{proof}

Now, for $i,j\not= 1$ we have $B_{ij}\peq E_i+b_{ij}E_j$ and, applying (\ref{a1}) to the cyclicly  ordered triple $(E_1,E_i,E_j)$, we get $$B_{ij}\peq b_{1i}E_i+b_{1j}E_j,$$ for all $i,j=2,\dots ,2n$. Taking into account that
$$
B_{1j} \peq E_1+b_{1j}E_j \,\,\,\, \hbox{and} \,\,\,\, B_{i1} \peq E_i+ b_{i1}E_1 \peq E_1+b_{1i}E_i \, ,
$$
if we rescale our $E_s$ to $b_{1s}E_s$ (which simply means choosing a more suitable
representative for $W_s$), for each $s\not= 1$, we get
\begin{equation}
B_{ij}\peq E_i+E_j \, ,
\label{a2}\end{equation}
for all $i,j=1,2,\dots ,2n$.
\vskip .5cm
Let $\ell$ be a line intersecting $\B$ in exactly  $m<p$  points and set  $$S = \ell \cap \B = \{ B_1,B_2,\dots ,B_m  \}.$$ For a fixed white point $W\in \cal H $,  define the sets
$$
S_-=\{ W_1^-,W_2^-,\dots ,W_m^- \} \,\,\, \hbox{and} \,\,\, S_+= \{  W= W_1^+,W_2^+,\dots ,W_m^+  \} \,
$$
where $W_i^-$ is the unique white point other than $W$ on the line $<W,B_i>$ and $W_i^+$ the unique white point other than $W_1^-$ on the line $<W_1^-,B_i>$, $i=1,\dots ,m$. In this way, using appropriate white points,  $\cal H$ can be partitioned into pairs of blocks, each block of size $m$, as we will show below, so $m$ must be a divisor of $n$.

If $B=B_{ij}$ is the black point on the line $<W_i^+,W_j^->$, $i,j=1,2,\dots ,m$,
then, fixing coordinates
of the black points on $\ell$ so that $B_i=E_1^+ +E_i^-=E+E_i^-$:
$$
B = B_{ij} \peq E_i^++E_j^- \peq B_i- E_1^- + B_j - E_1^+ \peq \alpha B_i + \beta B_j - (E_1^- +
E_1^+) = \alpha B_i + \beta B_j + \gamma B_1,
$$
for some constants $\alpha , \beta , \gamma $. So $B$ is on the line $\ell$.  It follows that $S\cup S^-\cup S^+$ is a dual $3$-net of order $m$ and, by Theorem \ref{BKM}, $S^-\cup S^+$ is a set of $2m$ points on a conic $\Gamma $ and $m$ is the size of a subgroup of a group defined on $\cal H\setminus \ell$ ({\it cf.} Example \ref{ex1}). It follows that $m$ divides $p+1$ or $p-1$ according to the fact that $\Gamma$ has $0$  or $2$  points in common with $\ell$, respectively. According to this two cases, we say that $\cal H$ is of \it elliptic \rm or \it hyperbolic \rm type,  respectively. Note that the case $|{\cal H} \cap \ell|= 1$ cannot occur, otherwise $m$ should divide $p$.
\\[5pt] \indent Our last results can be summarized in the following lemma.
\begin{lemma} \label{ll1}
If a line $\ell$ intersects $\B$ in exactly $m<p$ \; points, then $m$ divides $n$. Moreover, $m$ divides $p-1$ or $p+1$ according to $\cal H$ is of hyperbolic or elliptic type w.r.t. $\ell$, respectively.
\end{lemma}

\begin{lemma} \label{lm5}
If a line $\ell$ intersects $\B$ in exactly $m<p$ points, then $m\le 4$.
\end{lemma}
\begin{proof}
Suppose $m\ge 5$ and consider (homogeneous) projective coordinates $(x:y:z)$ in $PG(2,p)$ in  a way that the equation of $\ell$ is $z=0$. If $\cal H$ is of hyperbolic  type w.r.t. $\ell$, then we can assume that the conic $\Gamma$ has equation $xy=z$ and, by Theorem \ref{BKM}, we may take
$$
S_+ = \{  (u:\frac{1}{u}:1) \,\,\, : \,\,\, u\in \F^* \,\, \hbox{and} \,\, u^m=1 \}.
$$
We can also assume that $S_+$ contains the point $E_0=(1,1,1)$ and the four distinct points of type  $E_1=(v,1/v,1)$, $E_{-1}=(1/v,v,1)$, $E_2=(v^2,1/{v^2},1)$, $E_{-2}=(1/{v^2},v^2,1)$, where $v$ is a primitive $m-$root of unity. Then the two different black points $B_1\peq E_{-1}+ E_1 = ((v+1)/v,(v+1)/v,2)$, $B_2\peq E_{-2}+ E_2 = ((v^2+1)/{v^2},(v^2+1)/{v^2},2)$ and the white point $W_0=[E_0]$ are on the line $x=y$, a contradiction.
\\
In case $\HH$ is of elliptic type w.r.t. $\ell$, we identify in the standard way the affine plane $AG(2,p) = PG(2,p) \setminus \ell$ with the field $F_{p^2}$. Recall that under this identification three
(different) points $a,b,c\in AG(2,p)$ are collinear iff $(a-c)^{p-1}=(b-c)^{p-1}$.   If we denote by $(x;1)$, with $x\in F_{p^2}$, the coordinates of affine points of $PG(2,p)$,  we can assume that the conic $\Gamma$ has equation $x^{p+1}=1$ and we may take
$$
S_+= \{ (x;1) \,\,\, :\,\,\, x^m=1  \} \, .
$$
Remark that every element of $S_+$ is forced to be an element of $\Gamma $ because $m$ divides $p+1$. Now we can assume that $S_+$ contains the points (corresponding to) $E_0=(1;1)$ and four distinct points of type  $E_1=(v;1)$, $E_{-1}=(1/v;1)$, $E_2=(v^2;1)$, $E_{-2}=(1/{v^2};1)$, where $v$ is a primitive $m$-th root of unity.  Again the claim is that we get a line with two black points and a white one. Actually, the points $W_0=[E_0]$,
$B_1 \peq E_{-1}+ E_1 = (1/v+v;2) = (v^p+v;2)$ and $B_2\peq E_{-2}+ E_2 = (1/{v^2}+v^2;2) = (v^{2p}+v^2;2)$ are on the line whose affine part has equation $x^p=x$.
\end{proof}

A $4$-set $\{i,j,k,l\}$ of indices, or $\{ E_i,E_j,E_k,E_l \}$ of points of $\HH$ is said to be {\em special} if $E_i+E_j+E_k+E_l={\bf 0}$. In this case, the white points $E_i,E_j,E_k,E_l$ determine exactly three black ones, namely $B_{ij}=B_{kl}$, $B_{jl}=B_{ik}$ and $B_{il}=B_{jk}$, and of course $\{ E_i,E_j,E_k,E_l \}$ is a generalized hyperfocused arc contained in $\cal H$.

\begin{lemma} \label{lm6}
Let $E_1,E_2,E_3,E_4,E_5,E_6$ be six distinct points of $\cal H$ such that
$$ E_1+E_2 \peq E_3+E_4 \peq E_5+E_6 \peq B \, .$$
Then, if $\{ E_1,E_2,E_3,E_4 \}$ and $\{ E_1,E_2,E_5,E_6 \}$ are special, $\{ E_3,E_4,E_5,E_6 \}$ is not.
\end{lemma}
\begin{proof}
If $E_1+E_2+E_3+E_4=E_1+E_2+E_5+E_6={\bf 0}$ then $E_3+E_4=E_5+E_6$ so $E_3+E_4+E_5+E_6\ne {\bf 0}$ since $p>2$
and $E_3+E_4\ne {\bf 0}$.
\end{proof}

We are now in the right position in order to prove our main result.

\begin{theorem} If $p$ is an odd prime and $\cal H$ is a hyperfocused arc of size $2n$, in $PG(2,p)$ then $n\le 2$.
\end{theorem}
\begin{proof}
Since $n>2,$ by Lemma \ref{lm6}, we can consider four points $E_i$ with $E_1+E_2\peq E_3+E_4$ but with
$E_1+E_2+E_3+E_4\ne{\bf 0}.$

Then we have $5$ different black points: $B_{12}=B_{34}$,
$B_{13}$, $B_{14}$, $B_{23}$ and $B_{24}$, moreover
$B_{12}=B_{34},B_{13}$, and $B_{24}$ are
 collinear, and also $B_{12}=B_{34},B_{14}$ and $B_{23}$ are collinear,
 and these two lines are different (for otherwise we would have a line with
 more than $4$ black points). We first show that $B_{13}\ne B_{24}$.
 If they were equal, then we would have $B_{13}\peq E_1+E_3\peq E_2+E_4\peq
 E_1+E_2+E_3+E_4\peq B_{12}$. The same argument shows that $B_{14}\ne B_{23}$.
 Next we show that $B_{12}=B_{34}$, $B_{13}$ and $B_{24}$ are collinear.
 This is clear since the vector $E_1+E_2+E_3+E_4$ (representing $B_{12}=B_{34}$)
 is a linear combination of $E_1+E_3$ and $E_2+E_4$.

After a suitable linear transformation we may put
$B_{12}=B_{34}=(0,0,1)$, $B_{13}=(1,0,0)$, $B_{24}=(1,0,1)$,
$B_{14}=(0,1,0)$ and $B_{23}=(0,1,1)$. Remark that since the lines
$x=0$, $y=0$ and $z=0$ contain (at least) two black points the
coordinates of white points are now all nonzero.
If we put $E_1=(a,b,c)$
then we can compute 'everything'. Indeed,
from $B_{12}$ we have
$E_2=(-a,-b,*)$ and  from $B_{14}$ we have $E_4=(-a,*,-c)$. Now by
$B_{24}$ we get  $E_2=(-a,-b,c-2a)$ and $E_4=(-a,b,-c)$. From
$B_{13}$ we find $E_3=(*,-b,-c)$ and from $B_{34}=B_{12}$ we find
$E_3=(a,-b,-c)$. Finally, from $B_{23}$ we have $a=b$, and hence
$E_1=(a,a,c)$, $E_2=(-a,-a,c-2a)$, $E_3=(a,-a,-c)$ and
$E_4=(-a,a,-c)$. Note that our configuration has an
$x\leftrightarrow y$ symmetry. We proceed to show that in fact one
of the two lines $x=0$ or $y=0$  is a $3$-secant, and the other one is a $4$-secant.

Let start by looking at the case that both lines are a $3$-secant.
We then find additional white points $E_5$ and $E_6$ on one conic
and $E_5'$ and $E_6'$ on the other: $(a-2c,a,2a-c)$ and
$(2c-a,-a,c)$ on the conic with respect to the
 line $y=0$ as well as $(a,a-2c,2a-c)$ and
$(-a,2c-a,c)$, on the conic with respect to $x=0$. Adding two white points
produces a black point, and we find black points
$(2a-2c,2a-2c,4a-2c)$ and $(2c-2a,2c-2a,2c)$. Since the only black point
on the line $x=y$ is $(0,0,1)$, then  $a=c$, i.e. $E_1=E_2,$ a contradiction.

Next we assume that the line $x=0$ is a $4$-secant and let $B$ be the fourth black point on it. We label the black points
on this line with $B_{12}=B_{34}=(0,0,1)=\infty$, $B_{14}=(0,1,0)=0$, $B_{23}=(0,1,1)=1$,
and $B=(0,1,b)=b,$ with
$b\ne 0,1,\infty.$
Let $W_5, \dots, W_8$ be the remaining white points on the conic determined by the line
$x=0$ and the point $W_1$, with $W_1,W_3,W_5,W_7$ forming one component, and the remaining
$W_i$ the other. For a suitable ordering of the $W_i$ we then find the following table
for $(W_i,W_j,B_{ij})$.
\[
\begin{matrix}
    & W_1 & W_3    & W_5 & W_7 \\
 W_2 &  \infty  & 1  &  0  &  b  \\
 W_4 &  0   & \infty & b  &  1 \\
 W_6 & 1 & b & \infty  & 0\\
 W_8 & b & 0 & 1 & \infty
\end{matrix}
\]
\noindent From $W_2+W_5\peq 0 =(0,1,0)$ we get $W_5=(a,u,2a-c)$.

\noindent From $W_8+W_3\peq 0= (0,1,0)$ we get $W_8=(-a,v,c)$.

\noindent From $W_5+W_8\peq 1= (0,1,1)$ we get $v=2a-u$.

\noindent From $W_5+W_4\peq b=(0,1,b)$ we get $(1,b)\peq (a+u,2a-2c)$.

\noindent From $W_8+W_1\peq b=(0,1,b)$ we get $(1,b)\peq (3a-u,2c)$.

Since $a\ne 0$ we can compute $u=3a-4c$.

Finally $b=(0,1,b)\peq W_4+W_5=(0,a+u,2a-2c)=(0,4a-4c,2a-2c)$, so $b=1/2$.

Next we consider the case that both lines ($y=0$ and $x=0$) are a $4$-secant: because of
$x,y$ symmetry we then also find points $W_j'$, $j=5,\dots,8$, in particular
we find $W_5'=(3a-4c,a,2a-c)$ and a hence a black point $B_{55'}\peq (4a-4c,4a-4c,4a-2c)$ on the
line $x=y$, but the only black point on this line is $(0,0,1)$, hence $a=c$ and
the whole thing collapses.

So we find that one line is a $4$-secant, and the other a $3$-secant.

So summarizing: if $E_1+E_2\peq E_3+E_4$ but their sum
is nonzero, then the points $B_{12}=B_{34},B_{14},B_{23}$ are collinear and
so are the points $B_{12}=B_{34}, B_{13}, B_{24}$, one of these lines
is a three-secant, and the other a four secant, with fourth point $B$, determining $1/2$, that is
if $B_{12}=B_{34}=\infty$, $B_{14}=0$, $B_{23}=1$, is the four-secant, then $B=1/2$. Note that the
role of $B_{12}=B_{34}$ is special, but $B_{23}$ and $B_{14}$ can be interchanged,
indeed, the (Moebius)-transformation fixing $\infty$ and interchanging $1$ and $0$
is $x\mapsto 1-x$ and this fixes $1/2$, as it should.

We have seen that a non-special
fourtuple $\{W_i,W_j,W_k,W_l\}$ with $B\peq E_i+E_j\peq E_k+E_l$ produces a 3-secant $\ell$
and a 4-secant $m$
intersecting in $B$, but more important:
we can find back the pairs $(W_i,W_j)$ and $(W_k,W_l)$ from the triple $(\ell,m,B)$.
The four black points on the
4-secant $m$ fall into three groups: the intersection point $B$, a pair $C_1,C_2$, and the fourth point $D$,
the midpoint of the other two when we put the intersection point at infinity, or in other words
the cross ratio $(B,C_1;C_2,D)=(B,C_2;C_1,D)=1/2$. Fixing coordinates
as we did we found that $W_1$ and $W_2$ are on
the line $x=y$ and $W_3$ and $W_4$ are on the line $x=-y$, and these lines
of course contain no further white points.

Now start with a fixed $3$-secant on a fixed black point $B$. The $3$-secant
determines a partition of the $n$ lines connecting $B$ with the white points in
$n/3$ triples. Every triple determines $3$ pairs of pairs $(W_i,W_j)$ and $(W_k,W_l)$
and the fourtuple $(W_i,W_j,W_k,W_l)$ is not special, because the sum of the
six white point $W_i,W_j,W_k,W_l,W_m,W_n$ on the conic determined by the $3$-secant and
any of the $W$'s add up to zero (that is, the corresponding $E$'s do).

So this pair of pairs determines a configuration where one of the lines is the
fixed $3$-secant, and the other a $4$-secant. Now, every $4$-secant
(intersecting our fixed $3$-secant in the fixed black point $B$) gives (at most) $1$ configuration, as we saw above, so we find at least $(n/3)\times 3=n$
different $4$-secants through $B$, containing together (apart from $B$) at
least $3n$ black points, contradiction.

So all four tuples are special, but this means $2n\le 4$.
\end{proof}
\section{Generalizing further}
Generalized hyperfocused arcs were introduced as  a generalization of hyperfocused arcs, but this is not how we hit upon them. In fact we see them as a special case of a configuration whose study is motivated
by the (strong) cylinder conjecture \cite{B1}. This conjecture states that a set $C$ of $q^2$ points in $AG(3,q),$ the affine $3$--space over $F_q,$
that intersects every plane in $0$ mod $q$ points must be a cylinder, the union of $q$ parallel lines.
The ordinary conjecture states this for $q$ is prime, the strong version for prime powers.

In an attempt to prove this conjecture we were led to the study of configurations in $PG(2,p)$
consisting of a set $\mathcal W$ of $k\le (p+1)/2$ \, 'white' points, and a {\em multiset}
$\B$ of $k-1$ 'black' points with the property that every line containing $m>0$ white points contains
exactly $m-1$ black points (counted with multiplicity). The connection is as follows.
\\ [5pt]
Embed $AG(3,q)$ in $PG(3,q)$ and let $P$ be a point of $AG(3,q)$ not in $C$. Project $C$ on to a plane $PG(2,q)$ not containing $P$.
We get a multiset of $q^2$ points in $PG(2,q)$ with a multiple of $q$ points on every line. \newline
Let the \it weight \rm of a line $\ell$ in $PG(2,q)$ be $k-1 ,$ if the plane $\langle P,\ell\rangle$ has
$kq$ points.
Consider a line of weight $-1$.
Clearly all points of it have weight 0, that is, they don't occur in the multiset, or equivalently, they occur in the multiset with multiplicity zero.

If we add up the weights
of the lines through a zero-point, we get $-1$. So points on the intersection of
two ($-1$)-lines must be on a weight $>0$ line as well. More precisely, if
$Q$ is a point on $k$ ($-1$)-lines, then the weights of the positive lines through $Q$
add up to $k-1$.

Now, dualizing,
we get two disjoint point-sets: a set of \it white points, \rm corresponding to the duals of the ($-1$)-lines, and a multi-set of \it black points \rm corresponding to the duals of the
positive weight lines, the multiplicity of a point being the weight of its corresponding line. Every line
that contains more than one white point contains the appropriate number of black points.
Every line having exactly one white point has no black points. The total number of
black points is one less than the number of white points.
In the example of the cylinder all black
and white points are collinear; so our problem is to find (and exclude) the other configurations. Remark that we may choose $P$ in such a way that the number of lines with weight
different from 0 is at most $q$ (there are at most $q^2-q$  planes that intersect $C$
in a number of points different from $q$, choose $P$ such that it is on relatively few of them).
\\ [5 pt]
\indent In case the set $\mathcal W$ is an arc, the multiset
$\B$ is an ordinary set and we are looking at a generalized hyperfocused arc. Our main result therefore is that this
situation essentially does not occur.  To classify these configurations in general seems to be hopeless
but the prime case could be doable (and might settle the cylinder conjecture!).
Here we give the relevant (that is, with at most $(p+1)/2$ white points) examples we know in planes of prime order. In general many more examples exist, also in planes of prime order.

\begin{eg} {\bf\small (White and black points are collinear)}
All white points are collinear, black points are arbitrary other
points on this line, the right number of them. This is the only example
coming from cylinders.
\label{coll1}\end{eg}

\begin{eg} {\bf\small (White points on two lines)}
In  $AG(2,p)$ consider as white points $(a,0)$ and $(0,b)$ where
$a$ and $b$ are in a (multiplicative) subgroup of $GF(p)^*$ of
order $n$ say (or in a coset). Take black points at infinity in the points
$(a:-b:0)=(1:-b/a:0)$, and take the origin with multiplicity $n-1$.

This example is characterized by the property that the white points are
contained in the union of two lines, for example by using the earlier
mentioned result on 3-nets \cite{BKM} (but in its general form).

\label{sbgr1}\end{eg}

\begin{eg} {\bf\small (White points form an arc)}
The white points form a $4$-arc, and there
are $3$ black points, the diagonal points. This is the example classified in this note.
\label{2narc1}\end{eg}
 \section*{ Acknowledgement } This research was supported by the  \it Mathematics and Physics  Department \rm  of the Seconda Universit\`a
degli Studi di Napoli,
by \it GNSAGA \rm of the Italian Istituto Nazionale di Alta Matematica and
by Italian national research project \it Strutture geometriche,
combinatoria e loro applicazioni \rm (COFIN 2008).



\vskip 2cm
Addresses of the authors:\\

Aart Blokhuis\\
Department of Mathematics and Computer Science\\
Eindhoven University of Technology\\
P. O. Box 513\\
5600 MB Eindhoven\\
The Netherlands\\
(aartb@win.tue.nl)\\

Giuseppe Marino, Francesco Mazzocca\\
Dipartimento di Matematica e Fisica\\
Seconda Università degli Studi di Napoli\\
Via Vivaldi 43\\
81100 Caserta\\
Italy\\
(giuseppe.marino@unina2.it, francesco.mazzocca@unina2.it)\\

\end{document}